\documentclass[oneside, a4paper, 10pt]{article}
\usepackage[colorlinks = true,
            linkcolor = black,
            urlcolor  = blue,
            citecolor = black]{hyperref}
\usepackage{amsmath, amsthm, amssymb, amsfonts, amscd}
\usepackage[left=3cm,right=3cm,top=3cm,bottom=3cm,a4paper]{geometry}
\usepackage{mathtools}
\usepackage{thmtools}
\usepackage{graphicx}
\usepackage[titletoc]{appendix}
\usepackage{cleveref}
\usepackage{mathrsfs}
\usepackage{titling}
\usepackage{url}
\usepackage[nosort, noadjust]{cite}
\usepackage{enumitem}
\usepackage{tikz-cd}
\DeclareGraphicsExtensions{.pdf,.png,.jpg}

\theoremstyle{definition}
\newcommand{\disc}{{\operatorname{disc}}}
\newcommand{\id}{{\operatorname{id}}}

\newcommand{\Cl}{{\operatorname{Cl}}}
\newcommand{\Pic}{\operatorname{Pic}}

\newcommand{\Hom}{\operatorname{Hom}}
\newcommand{\End}{\operatorname{End}}
\newcommand{\GL}{\operatorname{GL}}

\newcommand{\Gr}{\operatorname{Gr}}

\newcommand{\Isog}{\operatorname{Isog}}

\newcommand{\Stab}{\operatorname{Stab}}
\newcommand{\Z}{\mathbb{Z}}
\newcommand{\Q}{\mathbb{Q}}
\newcommand{\R}{\mathbb{R}}
\newcommand{\F}{\mathbb{F}}
\newcommand{\A}{\mathbb{A}}
\newtheorem{theorem}{Theorem}[section]
\newtheorem{proposition}[theorem]{Proposition}
\newtheorem{corollary}[theorem]{Corollary}
\newtheorem{lemma}[theorem]{Lemma}

\newtheorem*{conjecture*}{Conjecture}
\newtheorem{remark}[theorem]{Remark}

\newcommand{\changeurlcolor}[1]{\hypersetup{urlcolor=#1}} 

\usepackage{array}
\makeatletter
\newcommand{\thickhline}{%
    \noalign {\ifnum 0=`}\fi \hrule height 1pt
    \futurelet \reserved@a \@xhline
}
\newcolumntype{"}{@{\hskip\tabcolsep\vrule width 1pt\hskip\tabcolsep}}
\makeatother

\title{\large{\textbf{ON THE LOWER BOUND OF THE NUMBER OF ABELIAN VARIETIES OVER $\F_p$}}}
\author{\normalsize{JUNGIN LEE}}
\date{}

\usepackage{fancyhdr}

\newcommand\shorttitle{ON THE LOWER BOUND OF THE NUMBER OF ABELIAN VARIETIES OVER $\F_p$}
\newcommand\authors{JUNGIN LEE}

\usepackage{sectsty}
\sectionfont{\large}
\subsectionfont{\normalsize}

\begin{document}
\maketitle

\vspace{-12mm}

\begin{abstract}
In this paper, we prove that the number $B(p,g)$ of isomorphism classes of abelian varieties over a prime field $\F_p$ of dimension $g$ has a lower bound $p^{\frac{1}{2}g^2(1+o(1))}$ as $g \rightarrow \infty$. 
This is the first nontrivial result on the lower bound of $B(p,g)$. 
We also improve the upper bound $2^{34g^2}p^{\frac{69}{4} g^2 (1+o(1))}$ of $B(p,g)$ given by Lipnowski and Tsimerman (Duke Math. J. 167:3403-3453, 2018) to $p^{\frac{45}{4} g^2(1+o(1))}$. 
\end{abstract}


\section{Introduction} \label{Sec1}

Let $p$ be a prime and $q=p^a$ be its power. Counting the number of isomorphism classes of abelian varieties over a finite field $\F_q$ of dimension $g$ can be divided into two parts: (1) classifying the isogeny classes of abelian varieties over $\F_q$ of dimension $g$ and (2) counting the size of each isogeny class. By Honda-Tate theorem \cite{Hon68, Tat66}, the isogeny classes of abelian varieties over $\F_q$ of dimension $g$ are in bijection with Weil $q$-polynomials of degree $2g$, which gives a solution to the first part. One can deduce that the number $m_q(g)$ of isogeny classes of abelian varieties over $\F_q$ of dimension $g$ is given by
\begin{equation} \label{eq1}
m_q(g) = q^{\frac{1}{4}g^2(1+o(1))}
\end{equation}
(as $g \rightarrow \infty$) by counting the number of Weil $q$-polynomials of degree $2g$.

For the second part, there is a linear-algebraic description of the set of abelian varieties in each isogeny class due to Kottwitz \cite{Kot90}. (See Section \ref{Sub22} for details.) However, it is very hard to compute the size of each isogeny class in general. Explicit formulas for the size of the isogeny classes are only known for some simple cases, such as elliptic curves \cite[Theorem 4.6]{Sch87} or simple abelian surfaces correspond to Weil $q$-number $\pi = \sqrt{q}$ \cite[Theorem 4.4]{XY20}.

Alternatively, one can try to estimate the number $B(q,g)$ of isomorphism classes of abelian varieties over $\F_q$ of dimension $g$ without its explicit formula. For the case $q=p$, Lipnowski and Tsimerman \cite{LT18} obtained an upper bound 
$$
B(p,g) \leq p^{\frac{17}{2} g^2 (1+o(1))}
$$ 
for a fixed $p$ as $g \rightarrow \infty$. Note that there are some minor errors in their proof, which will be corrected in Section \ref{Sub31} of this paper. A (corrected) main result of \cite{LT18} is the following theorem:

\begin{theorem} \label{thm11}
(\cite{LT18}, Theorem 0.1) Let $B(p,g)$ be the number of isomorphism classes of abelian varieties over $\F_p$ of dimension $g$ for a fixed $p$. Then
\begin{equation} \label{eq2}
B(p,g) \leq 2^{34g^2}p^{\frac{69}{4} g^2 (1+o(1))}. 
\end{equation}
\end{theorem}

It is natural to investigate the lower bound of $B(q,g)$. Since there is a trivial bound 
$$
B(q,g) \geq m_q(g) = q^{\frac{1}{4}g^2(1+o(1))},
$$
a lower bound of $B(q,g)$ can be meaningful only if it is larger than $q^{\frac{1}{4}g^2}$. Up to our knowledge, there had not been known any nontrivial lower bound of $B(q,g)$ in the literature. 

The purpose of this paper is to provide a nontrivial lower bound of $B(q,g)$ for the case $q=p$. Denote the number of isomorphism classes of simple (precisely, $\F_p$-simple) abelian varieties over $\F_p$ of dimension $g$ by $B^{\text{sim}}(p,g)$. The main result of this paper is the following theorem.

\begin{theorem} \label{thm13}
(Theorem \ref{thm411}) 
\begin{equation} \label{eq4}
B(p,g) \geq B^{\text{sim}}(p,g) \geq p^{\frac{1}{2}g^2(1+o(1))}. 
\end{equation}
\end{theorem}

\begin{remark} \label{rmk1x}
Since
$$
p^{\frac{1}{2}g^2(1+o(1))} = p^{\frac{1}{4}g^2} p^{\frac{1}{4}g^2(1+o(1))},
$$
a lower bound above is $p^{\frac{1}{4}g^2(1+o(1))}$ times larger than $p^{\frac{1}{4}g^2}$. Thus this is the first nontrivial result in this context. 
\end{remark}

This paper is organized as follows. 
We introduce some definitions and notations and review the description of the set of abelian varieties in a single isogeny class in terms of lattices in Dieudonné modules and Tate modules in Section \ref{Sec2}. 
Section \ref{Sec3} is devoted to the improvement of the upper bound of $B(p,g)$ provided in Theorem \ref{thm11}. More precisely, in Section \ref{Sub31}, we briefly summarize how Lipnowski and Tsimerman bounded the size of each isogeny class using the description given in Section \ref{Sec2}. After that, we provide an improvement of their result in Section \ref{Sub32} and \ref{Sub33}. 
The improvement can be done by modifying the proof of \cite{LT18}. 

\begin{theorem} \label{thm12}
(Theorem \ref{thm34}) 
\begin{equation} \label{eq6}
B(p,g) \leq p^{\frac{45}{4} g^2 (1+o(1))}. 
\end{equation}
\end{theorem}

In Section \ref{Sec4}, we provide a lower bound of $B^{\text{sim}}(p,g)$ (and consequently, $B(p,g)$) as explained above. The key ingredients of the proof of Theorem \ref{thm13} are the followings: 
\begin{enumerate}[label=(\roman*)]

\item (Corollary \ref{cor22}) A formula for the size of an isogeny class of simple abelian varieties over $\F_p$ associated to a Weil $p$-number $\pi$ as a sum of class numbers of orders in $\Q (\pi)$;

\item (\cite[Theorem 2]{Sta74}) Stark's inequality relating the class numbers of a CM-field $E$ and its maximal totally real subfield $E^+$ with the discriminants of $E$ and $E^+$;

\item (\cite[Theorem 2.3]{Lee}) Asymptotic information for irreducible Weil $p$-polynomials (equivalently, Weil $p$-polynomials associated to simple abelian varieties over $\F_p$);

\item (Proposition \ref{prop44}) For any $\varepsilon>0$ and a monic polynomial $f$ of degree $2g$, $\left | f(x) \right | \geq p^{g^2(1-\varepsilon)}$ for a large portion of $x$ in the interval $\left [ -\frac{2p^{\frac{g}{2}}}{g}, \frac{2p^{\frac{g}{2}}}{g} \right ]$ if $g$ is sufficiently large;

\item (Proposition \ref{prop410}) Detailed investigations on the discriminants and class numbers of non-maximal orders in $\mathcal{O}_{E}$ and $\mathcal{O}_{E^+}$.


\end{enumerate}


\section{Abelian varieties in each isogeny class} \label{Sec2}

\subsection{Definitions and notations} \label{Sub21}

In this section we provide a list of notation which will be used throughout the paper. Some of the notation are from \cite{LT18}. 

\begin{itemize}

\item Let $p$ be a prime, $q=p^a$ be a power of $p$ and $\F_q$ be a finite field with $q$ elements. Also let $\Z_q := W(\F_q)$ (Witt ring of $\F_q$) and $\Q_q$ be the fraction field of $\Z_q$. Note that $\Q_q$ is the unique unramified extension of $\Q_p$ of degree $a$. 

\item For an abelian variety $A_0$ over $\F_q$, denote the set of isomorphism classes of abelian varieties over $\F_q$ isogenous to $A_0$ by $\Isog (A_0)$. Similarly, denote the set of isomorphism classes of (simple) abelian varieties over $\F_q$ associated to a Weil $q$-number $\pi$ by $\Isog (\pi)$. 

\item For two abelian varieties $A$ and $B$ over $\F_q$, denote $\Hom^0(A,B) := \Hom_{\F_q}(A,B) {\otimes}_{\Z} \Q$ and $\End^0(A) := \End_{\F_q}(A) {\otimes}_{\Z} \Q$. 

\item An element $f \in \Hom^0(A,B)$ is called a \textit{quasi-isogeny} if $nf \in \Hom(A,B)$ is an isogeny for some integer $n$. Two quasi-isogenies $A \overset{f}{\longrightarrow} A_0$ and $B \overset{g}{\longrightarrow} A_0$ are isomorphic if there are two quasi-isogenies $A \overset{h}{\longrightarrow} B$ and $B \overset{k}{\longrightarrow} A$ such that $kh = \id_A$, $hk=\id_B$, $gh=f$ and $fk=g$. 
Denote the set of isomorphism classes of quasi-isogenies $A \overset{f}{\longrightarrow} A_0$ by $\widetilde{\Isog}(A_0)$. 

\item For an abelian variety $A_0$ over $\F_q$ and a prime $\ell \neq p$, let $T_{\ell}(A_0)$ be the $\ell$-adic Tate module of $A_0$ and $V_{\ell}(A_0) := T_{\ell}(A_0) \otimes_{\Z_{\ell}} \Q_{\ell}$. Also let $D(A_0)$ be the covariant Dieudonné module of $A_0$ (which is a module over the Dieudonné ring $D_{\F_q} := W(\F_q)\left \{ F, V \right \}/(FV-p)$) and $D^0 (A_0) := D (A_0) \otimes_{\Z_q} \Q_q$.

\item Let $W$ be a finite-dimensional vector space over a field $K$ and $g \in \GL_K(W)$. Define
$$
d'(g \mid W) := \prod_{\lambda \neq \mu} (\lambda - \mu)
$$
where $\lambda$, $\mu$ run over all pairs of distinct roots (counted with multiplicity) of the characteristic polynomial of $g$. 

\end{itemize}

\subsection{Linear-algebraic description of abelian varieties in an isogeny class} \label{Sub22}

To estimate the size of the set $\Isog (A_0)$ for a fixed abelian variety $A_0$ over $\F_q$, a linear-algebraic description of $\Isog (A_0)$ is needed. We closely follow the exposition of \cite[Section 3]{LT18}. Denote
\begin{equation*}
\begin{split}
X_p & := \left \{ M : M \subset D^0(A_0) \text{ is an } \left \langle F,V \right \rangle \text{-stable } \Z_q \text{-lattice} \right \} \\
X_{\ell} & := \left \{ M : M \subset V_{\ell}(A_0) \text{ is an } \text{Frob}_q \text{-stable } \Z_{\ell} \text{-lattice} \right \}\\
X^p & :=  \prod_{\ell \neq p} {' X_{\ell}} = \left \{ \prod_{\ell \neq p} L_{\ell} : L_{\ell} \in X_{\ell} \text{ and } L_{\ell}=T_{\ell}(A_0) \text{ for almost all } \ell \right \}.
\end{split}
\end{equation*}
Then the map
$$
\widetilde{\Isog}(A_0) \rightarrow X_p \times X^p
$$
given by 
$$
(A \overset{f}{\longrightarrow} A_0) \mapsto f_* D(A) \times \prod_{\ell \neq p} f_* T_{\ell}(A)
$$
is a bijection, and this induces a bijection
$$
\Isog(A_0) \rightarrow \End^0(A_0)^{\times} \setminus X_p \times X^p.
$$

When $A_0$ is a simple abelian variety over $\F_p$, there is an alternative description of $\Isog(A_0)$. Suppose that $A_0$ is a simple abelian variety over $\F_p$ associated to a Weil $p$-number $\pi \neq \sqrt{p}$. Then $E = \End^0(A_0)$ is a CM-field of degree $2 \dim A$ and is equal to $\Q(\pi)$. 

\begin{proposition} \label{prop21}
(\cite{Wat69}, Theorem 6.1) Let $A_0$ be as above and denote $R := \Z [\pi, p \pi^{-1}] \subset \mathcal{O}_E$. \\
(a) The endomorphism rings of abelian varieties isogenous to $A_0$ are exactly the orders in $E$ containing $R$. \\
(b) For each order $R'$ in $E$ containing $R$, the isomorphism classes of abelian varieties with endomorphism ring $R'$ correspond bijectively to the isomorphism classes of lattices in $E$ with order $R'$. 
\end{proposition}

See \cite[Theorem 5.1]{XYY16} for the generalization to non-simple case. Note that the result stated above is enough for our purpose. By Proposition \ref{prop21}, there is a bijection between $\Isog (A_0)$ and the ideal class monoid $\text{ICM}(R)$ of $R$ (see \cite[Definition 3.1]{Mar18}). The next corollary is the starting point of the proof of Theorem \ref{thm13}. 

\begin{corollary} \label{cor22}
$$
\left | \Isog (A_0) \right | = \left | \Isog (\pi) \right | \geq \sum_{R \subset B \subset \mathcal{O}_E} h(B)
$$
where $B$ runs through the orders of $E$ containing $R$ and $h(B)$ is the class number of $B$. In particular, 
$$
\left | \Isog (A_0) \right | \geq h(R).
$$
\end{corollary}

\begin{proof}
By the definition of $\text{ICM} (R)$, we have $\displaystyle \text{ICM} (R) \supset \bigsqcup_{R \subset B \subset \mathcal{O}_E} \Pic (B)$. 
\end{proof}

\begin{remark} \label{rmk23}
By  \cite[Proposition 3.7]{Mar18}, the equality 
$$
\left | \Isog (\pi) \right | = \sum_{R \subset B \subset \mathcal{O}_E} h(B)
$$
holds if and only if $R$ is Bass, i.e. every order $B$ is Gorenstein. 
\end{remark}


\section{Upper bound of $B(p,g)$} \label{Sec3}

\subsection{Bouding the size of $\Isog(A_0)$} \label{Sub31}

Starting from the bijection
$$
\Isog(A_0) \rightarrow \End^0(A_0)^{\times} \setminus X_p \times X^p
$$
in Section \ref{Sub22}, Lipnowski and Tsimerman \cite{LT18} obtained an upper bound of $\left | \Isog (A_0) \right |$. In this section we summarize their strategy and result with some corrections. Let $\mathbf{G}$ be the algebraic group over $\Q$ defined by
$$
\mathbf{G} (R) := (\End^0(A_0) \otimes_{\Q} R)^{\times}
$$
for every $\Q$-algebra $R$. The group $\mathbf{G} (\Q) = \End^0(A_0)^{\times}$ acts on $X_p \times X^p$ through $\mathbf{G} (\A ^{\text{fin}})$ where $\A ^{\text{fin}}$ is the finite adele ring of $\Q$. Let
$$
N := \left | \mathbf{G}(\A^{\text{fin}}) \setminus X_p \times X^p \right |
$$
and $L_1, \cdots, L_N$ be the orbit representatives for the action of $\mathbf{G}(\A^{\text{fin}})$ on $X_p \times X^p$. Then
\begin{equation} \label{eq7}
\End^0(A_0)^{\times} \setminus X_p \times X^p \cong \bigsqcup_{i=1}^{N} \mathbf{G}(\Q) \setminus \mathbf{G}(\A^{\text{fin}}) / \Stab_{\mathbf{G}(\A^{\text{fin}})} (L_i).
\end{equation}
From now on, assume that $q=p$ is a prime. Let $V = V_{\ell}(A_0)$ if $\ell \neq p$, $V = D^0(A_0)$ if $\ell=p$ and $\gamma \in \GL(V)$ denote the Frobenius element. Then the value $d'(\gamma \mid V) \in \Z$ is independent of the choice of $\ell$, which is denoted by $d'(A_0)$. Let 
$$
\chi_{\gamma} = f_1^{n_1} \cdots f_{j_{\ell}}^{n_{j_{\ell}}} \:\: (f_i \in \Z_{\ell}[x])
$$ 
be the decomposition of the characteristic polynomial of $\gamma$ into irreducible factors. Then $V = \oplus_{i=1}^{j_{\ell}}V_i$ where $V_i = \ker f_i(\gamma)$. 

An upper bound of the size of $\Isog (A_0)$ comes from an upper bound of the number $N$ and an upper bound of the size of $\mathbf{G}(\Q) \setminus \mathbf{G}(\A^{\text{fin}}) / \Stab_{\mathbf{G}(\A^{\text{fin}})} (L_i)$ for each $i$.

\begin{enumerate}

\item For a prime $\ell$ (including $p$), $N = \prod_{\ell} N_{\ell}$ where
$$
N_{\ell} := \left |\mathbf{G} (\Q_{\ell}) \setminus X_{\ell}  \right |.
$$
Denote $V_{\leq i} := \oplus_{k \leq i} V_k$ and $\Gr_i V := V_{\leq i} / V_{\leq i-1}$. Let $\gamma_i$ be the restriction of $\gamma$ on $\Gr_i V$ and $N_{\ell i}$ be the number of orbits of $\mathbf{Z}_{\gamma_i}(\Q_{\ell})$ (=the centralizer of $\gamma_i$ in $\GL(\Gr_i V)$) acting on the collection of $\gamma_i$-stable lattices in $\Gr_i V$. In Section 3.3 of \cite{LT18}, the bound of $N$ is derived as follows. 
\begin{enumerate}
\item (\cite[Corollary 3.16]{LT18}) $N_{\ell i} \leq \left | d' (\gamma_i \mid (\Gr _iV)) \right |_{\ell}^{-4}$.

\item (\cite[Corollary 3.17]{LT18}) $N_{\ell} \leq \left | d'(\gamma \mid V) \right |_{\ell}^{-4} = \left | d'(A_0) \right |_{\ell}^{-4}$.

\item (\cite[Corollary 3.18]{LT18}) $N= \prod_{\ell} N_{\ell} \leq \left | d'(A_0) \right |_{\infty}^4 \leq (2 \sqrt{p})^{4 \cdot \binom{2g}{2} \cdot 2}$. (The bound is written in \cite{LT18} as $(2 \sqrt{p})^{4 \cdot \binom{2g}{2}}$, but it should be corrected to $(2 \sqrt{p})^{4 \cdot \binom{2g}{2} \cdot 2}$.)
\end{enumerate}

\item There exists a same mistake as above in the proof of \cite[Corollary 3.22]{LT18}. Since the inequality (25) of \cite{LT18} comes from
\begin{equation} \label{eq15}
\# ( K/ \Stab_{\mathbf{G}(\A^{\text{fin}})} (\mathbb{L}) ) \leq \left | d'(A_0) \right |_{\infty}^4 \exp(o(g^2)),
\end{equation}
it should be corrected to
\begin{equation} \label{eq8}
\# ( K/ \Stab_{\mathbf{G}(\A^{\text{fin}})} (\mathbb{L}) ) \leq (2 \sqrt{p})^{4 \cdot \binom{2g}{2} \cdot 2 (1+o(1))}.
\end{equation}
The upper bound in the inequality (31) of \cite{LT18} should be corrected to $(2 \sqrt{p})^{2 \binom{2d}{2}}$ so the inequality (32) of \cite{LT18} should be
\begin{equation} \label{eq9}
\begin{split}
\# \prod_{i=1}^{m} \Cl (\mathcal{O}_{K_i}) & \leq \left ( \prod_{i=1}^{m} (2 \sqrt{p})^{\binom{d_i}{2}} \left ( \binom{d_i}{2} \log (2 \sqrt{p}) \right )^{d_i-1} \right ) \cdot 500^g e^g \\
& \leq (2 \sqrt{p})^{2g^2} \cdot \binom{2g}{2}^{2g} \cdot (\log 2 \sqrt{p})^{2g} 500^g e^g \\
& = (2 \sqrt{p})^{2g^2(1+o(1))}.
\end{split}
\end{equation}

\end{enumerate}

Following the proof of \cite[Proposition 3.23]{LT18} with above inequalities, we obtain
\begin{equation} \label{eq10}
\begin{split}
\left | \Isog (A_0) \right | 
& \leq C_0 \cdot (2 \sqrt{p})^{4 \cdot \binom{2g}{2} \cdot 2} \cdot (2 \sqrt{p})^{4 \cdot \binom{2g}{2} \cdot 2 (1+o(1))} \cdot (2 \sqrt{p})^{2g^2(1+o(1))} \\
& = 2^{34g^2} p^{17g^2(1+o(1))}.
\end{split}
\end{equation}
Thus
\begin{equation} \label{eq11}
B(p, g) \leq m_p(g) 2^{34g^2} p^{17g^2(1+o(1))}  = 2^{34g^2} p^{\frac{69}{4}g^2(1+o(1))},
\end{equation}
which completes the proof of Theorem \ref{thm11}. We improve this bound in the rest of this section. 


\subsection{Bound on $\left | d'(A_0) \right |_{\infty}$} \label{Sub32}

First we can easily reduce the upper bound on $\left | d'(A_0) \right |_{\infty}$. 

\begin{lemma} \label{lem3a}
For $m$ complex numbers $a_1, \cdots, a_m$ with absolute values $1$, 
$$
\prod_{1 \leq i < j \leq m} \left | a_i - a_j \right | \leq m^{\frac{m}{2}}
$$
and the equality holds if and only if $m$ numbers are equidistributed on a unit circle.
\end{lemma}

\begin{proof}
See the answer in \cite{Iva12}.
\end{proof}

This gives the bound
\begin{equation} \label{eq12}
\left | d'(A_0) \right |_{\infty} \leq \left ((2g)^{\frac{2g}{2}} \sqrt{p}^{\binom{2g}{2}}  \right )^2 = (2g)^{2g} p^{\binom{2g}{2}} = p^{2g^2(1+o(1))}.
\end{equation}
By the same reason, $\left | D_K \right | \leq d^d \sqrt{p}^{2 \binom{d}{2}}$ in the equation (31) of \cite{LT18} and
\begin{equation} \label{eq13}
\# \prod_{i=1}^m \text{Cl}(\mathcal{O}_{K_i}) \leq p^{g^2(1+o(1))}
\end{equation}
in the equation (32) of \cite{LT18}.


\subsection{Bound on $N$} \label{Sub33}

The following proposition is a refinement of \cite[Corollary 3.17]{LT18}. 

\begin{proposition} \label{prop31}
For every $C>1$, there exist $A>0$ and $\ell_0$ (each is independent of $g$) such that $N_{\ell} \leq A^{2g} \left | d'(\gamma \mid V) \right |_{\ell}^{-C}$ for all prime $\ell$ and $N_{\ell} \leq  \left | d'(\gamma \mid V) \right |_{\ell}^{-C}$ for all $\ell > \ell_0$. 
\end{proposition}

\begin{proof}
By the proof of \cite[Corollary 3.7]{LT18}, 
$$
N_{\ell} \leq \ell^{\delta'} \prod_{i=1}^{j_{\ell}} N_{\ell i}
$$ 
and 
$$
\left ( \prod_{i=1}^{j_{\ell}} \left | d' (\gamma_i \mid (\Gr_i V)) \right |_{\ell}^{-C} \right ) \cdot \ell^{C \delta'} = \left | d'(\gamma \mid V) \right |_{\ell}^{-C}
$$
for some $\delta' \geq 0$. So it is enough to show that there exist $A>0$ and $\ell_0$, independent of $g$ such that 
$$
N_{\ell i} \leq A \left | d' (\gamma_i \mid (\Gr_i V)) \right |_{\ell}^{-C}
$$ 
for all $i, \ell$ and 
$$
N_{\ell i} \leq  \left | d' (\gamma_i \mid (\Gr_i V)) \right |_{\ell}^{-C}
$$ 
for all $i$ and $\ell > \ell_0$. (Note that $\text{dim} V = 2g$ so $j_{\ell} \leq 2g$ for all $\ell$.)


\noindent For an integer $m \geq 0$, denote the set of partitions of $m$ by $\mathbf{P}(m)$ and let 
$$
P(m) := \left | \mathbf{P}(m) \right |.
$$ 
Also let $P(m,k)$ be the number of partitions of $m$ into $k$ parts. For $\lambda \in P(m)$, denote the length of $\lambda$ by $\ell(\lambda)$. 
Combining the inequality
$$
\text{length}_R (R^n/L) \leq \frac{n \delta}{d}
$$
in \cite[p. 3421]{LT18} and inequalities (14) and (15) of \cite{LT18}, $N_{\ell i}$ is bounded by the value of
$$
f(\ell, n, \delta, d) := {\ell}^{\frac{1}{2}n(n-1) \delta} \sum_{0 \leq b \leq \frac{\delta n}{d}} \sum_{a_1 + \cdots + a_n=b} \prod_{i=1}^{n} \left ( \sum_{\lambda_i \in \mathbf{P}(a_i)} (\ell^d)^{a_i - \ell(\lambda_i)} \right ).
$$
for some $n, \delta, d \geq 1$ satisfying 
$$
\left | d' (\gamma_i \mid (\Gr_i V)) \right |_{\ell}^{-1} = \ell^{n^2 \delta}.
$$
Since
\begin{equation*}
\begin{split}
f(\ell, n, \delta, d) & = {\ell}^{\frac{1}{2}n(n-1) \delta} \sum_{0 \leq db \leq \delta n} \sum_{da_1 + \cdots + da_n=db} \prod_{i=1}^{n} \left ( \sum_{d \lambda_i \in \mathbf{P}(da_i)} \ell^{da_i - \ell(d \lambda_i)} \right ) \\
& \leq f(\ell, n, \delta, 1)
\end{split}
\end{equation*}
for any $d \geq 1$, $N_{\ell i}$ is bounded by $f(\ell, n, \delta, 1)$. Denote 
$$
g(\ell, n, b) := \sum_{a_1 + \cdots + a_n=b} \prod_{i=1}^{n} \left ( \sum_{\lambda_i \in \mathbf{P}(a_i)} \ell^{a_i - \ell(\lambda_i)} \right )
$$
for simplicity. \\


(1) $n \geq 2$ : By the proof of \cite[Corollary 3.16]{LT18}, 
$$f(\ell, n, \delta, 1) \leq 
\ell^{\frac{1}{2}n(n+1) \delta} 2^{3 \delta n}$$
and if $\ell^{n-1} \geq 64$, then
$$\ell^{\frac{1}{2}n(n+1) \delta} 2^{3 \delta n} \leq \ell^{n^2 \delta} = \left | d' (\gamma_i \mid (\Gr_i V)) \right |_{\ell}^{-1}. $$
By Hardy-Ramanujan formula \cite{HR18}, there exists $M>0$ such that $P(m) \leq M \cdot 2^{\frac{m}{4}}$ for all $m \geq 0$. Then 
\begin{equation*}
\begin{split}
g(\ell, n, b) & \leq \sum_{a_1 + \cdots + a_n=b} \ell^{b} P(a_1) \cdots P(a_n) \\
& \leq \ell^{b} M^n 2^{\frac{b}{4}} \binom{b+n-1}{n-1} \\
& \leq \ell^{b} M^n \ell^{\frac{b}{4}}(b+1)^{n-1}
\end{split}
\end{equation*}
so
\begin{equation*}
\begin{split}
f(\ell, n, \delta, 1) & = {\ell}^{\frac{1}{2}n(n-1) \delta} \left ( 1 + \sum_{b=1}^{\delta n} g(\ell, n, b) \right ) \\
& \leq {\ell}^{\frac{1}{2}n(n-1) \delta} \cdot  {\ell}^\frac{5 \delta n}{4} M^n (\delta n +1)^{n-1} (\delta n +1) \\
& \leq \ell^{\frac{7}{8}n^2 \delta}M^n (\delta n+1)^n. 
\end{split}
\end{equation*}
Choose $\delta_0 >0 $ such that $\displaystyle M(x+1) \leq 2^{\frac{1}{8}x}$ for every $x \geq \delta_0$. Then if $\delta \geq \delta_0$, 
\begin{equation*}
\begin{split}
f(\ell, n, \delta, 1) & \leq \ell^{\frac{7}{8}n^2 \delta} (M(\delta n +1))^n \\
& \leq \ell^{\frac{7}{8}n^2 \delta} (2^{\frac{1}{8} \delta n})^n \\
& \leq \ell^{n^2 \delta}.
\end{split}
\end{equation*}


\vspace{2mm}

(2) $n = 1$ : 
\begin{equation*}
\begin{split}
f(\ell, 1, \delta, 1) & =1+\sum_{b=1}^{\delta} \sum_{k=1}^b P(b,k) \ell^{b-k}  \\
& \leq 1+\sum_{b=1}^{\delta} \sum_{k=1}^b P(b) \ell^{b-k}  \\
& \leq \delta P(\delta) \ell^{\delta}. 
\end{split}
\end{equation*}
By Hardy-Ramanujan formula, there exists $N=N(C)>0$ such that 
$$
mP(m) \leq N \cdot 2^{\frac{C-1}{2}m}
$$ 
for all $m \geq 1$. Then for $\ell^{\delta} \geq N^{\frac{2}{C-1}}$, 
\begin{equation*}
\begin{split}
f(\ell, 1, \delta, 1) & \leq \delta P(\delta) \ell^{\delta}  \\
& \leq N \cdot \ell^{\frac{C-1}{2} \delta} \ell^{\delta}   \\
& \leq \ell^{C \delta}  = \left | d' (\gamma_i \mid (\Gr_i V)) \right |_{\ell}^{-C}.
\end{split}
\end{equation*}


\vspace{2mm}

Now consider the finite set
$$
S := \left \{ (\ell, n, \delta) \mid n \geq 2, \, \ell^{n-1}<64, \delta < \delta_0  \right \} \cup
\left \{ (\ell, 1, \delta) \mid \ell^{\delta} < N^{\frac{2}{C-1}}  \right \}
$$
and let
$$
A := \text{max}\left \{ f(\ell, n, \delta, 1) \mid (\ell, n, \delta) \in S \right \}
$$
and
$$
\ell_0 := \text{max}\left \{ 64, N^{\frac{2}{C-1}} \right \}.
$$
Note that each of $S$, $A$ and $\ell_0$ depends only on $C>1$. Then 
$$
f(\ell, n, \delta, 1) \leq \left | d' (\gamma_i \mid (\Gr_i V)) \right |_{\ell}^{-C}
$$ 
for $(\ell, n, \delta) \notin S$ and 
$$
f(\ell, n, \delta, 1) \leq A
$$ 
for $(\ell, n, \delta) \in S$, so
$$
N_{\ell i} \leq f(\ell, n, \delta, 1) \leq A \left | d' (\gamma_i \mid (\Gr_i V)) \right |_{\ell}^{-C}.
$$
For $\ell > \ell_0$, $(\ell, n, \delta) \notin S$ so $N_{\ell i} \leq \left | d' (\gamma_i \mid (\Gr_i V)) \right |_{\ell}^{-C}$. 
\end{proof}

\begin{corollary} \label{cor32}
$N \leq p^{2g^2(1+o(1))}$. 
\end{corollary}

\begin{proof}
By proposition \ref{prop31} and (\ref{eq12}),
\begin{equation*}
\begin{split}
N & = \prod_{\ell} N_{\ell}   \\
& \leq \prod_{\ell \leq \ell_0} A^{2g} \left | d'(A_0) \right |_{\ell}^{-C} \cdot \prod_{\ell > \ell_0} \left | d'(A_0) \right |_{\ell}^{-C}  \\
& \leq A^{2g \ell_0} \left | d'(A_0) \right |_{\infty}^C  \\
& \leq p^{2Cg^2(1+o(1))}
\end{split}
\end{equation*}
for every $C>1$. Thus $N \leq p^{2g^2(1+o(1))}$. 
\end{proof}

Now we can prove the main result of this section. 

\begin{theorem} \label{thm34}
\begin{equation} \label{eq16}
B(p,g) \leq p^{\frac{45}{4} g^2 (1+o(1))}. 
\end{equation}
\end{theorem}

\begin{proof}
By the inequalities (\ref{eq15}), (\ref{eq12}) and (\ref{eq13}) with Corollary \ref{cor32}, we obtain
\begin{equation*}
\begin{split}
B(p,g) & \leq m_p(g)  \cdot N \cdot \left | d'(A_0) \right |_{\infty}^4 \exp{o(g^2)} \cdot p^{g^2(1+o(1))} \\
& \leq p^{(\frac{1}{4} + 2 + 8 + 1)g^2(1+o(1))} \\
& \leq p^{\frac{45}{4}g^2(1+o(1))}. 
\end{split}
\end{equation*}
\end{proof}


\section{Lower bound of $B(p,g)$} \label{Sec4}

Let $A_0$ be a simple abelian variety over $\F_p$ of dimension $g \geq 3$ (so it is associated to the Weil $p$-number $\pi \neq \sqrt{p}$). Then 
$$E := \End^0(A_0) = \Q (\pi)$$ 
is a CM-field of degree $2g$. Denote the maximal totally real subfield of $E$ by $E^+$. Let 
$$R := \Z [\pi, \overline{\pi}] = \Z [\pi, p \pi^{-1}]$$ 
and 
$$R^+ := \Z [\pi + \overline{\pi}].$$ 
Then $R$ and $R^+$ are orders in $E$ and $E^+$, respectively. Corollary \ref{cor22} says that
$$
\left | \Isog (\pi) \right | = \left | \Isog (A_0) \right | \geq h(R),
$$
so we need a lower bound of $h(R)$ for sufficiently many Weil $p$-numbers $\pi$ to give a lower bound of $B^{\text{sim}}(p,g)$ (and consequently, $B(p,g)$).

\subsection{Discriminant of $R$} \label{Sub41}

Roughly speaking, we prove that the absolute value of the discriminant of $R$ is sufficiently large for sufficiently many $\pi$ in this section. (See Corollary \ref{cor46} for a precise statement.) First we recall some notations and results in \cite[Section 2]{Lee}. For integers $a_1, \cdots, a_g$, let $F(a_1, \cdots, a_g)$ be a polynomial in $x$ defined by
$$
F(a_1, \cdots, a_g) := (x^{2g}+p^g)+a_1(x^{2g-1}+p^{g-1}x) + \cdots + a_{g-1}(x^{g+1}+px^{g-1})+a_gx^g.
$$
Let
\begin{equation*} 
\begin{split}
Y_{g}^1 & := \left \{ (a_1, \cdots, a_{g-1}) \in \Z^{g-1} : \left | \frac{a_i}{p^{i/2}} \right | \leq \frac{1}{g} \,\, (1 \leq i \leq g-1) \right \} \\
Y_g^2 & :=  \left \{ a_g \in \Z : \left | \frac{a_g}{2p^{g/2}} \right | \leq \frac{1}{g} \text{ and } \text{gcd} (a_g, p)=1 \right \} \\
Y_g & := Y_g^1 \times Y_g^2 \\
& =  \left \{ (a_1, \cdots, a_g) \in \Z^g : \left | \frac{a_g}{2p^{g/2}} \right | \leq \frac{1}{g}, \, \left | \frac{a_i}{p^{i/2}} \right | \leq \frac{1}{g} \,\, (1 \leq i \leq g-1) \text{ and } \text{gcd} (a_g, p)=1 \right \}.
\end{split}
\end{equation*}
Then by \cite[Lemma 2.1]{Lee}, $F(a_1, \cdots, a_g)$ is a Weil $p$-polynomial for any $(a_1, \cdots, a_g) \in Y_g$. This enables us to define
$$
Y_g^{\text{sim}} :=  \left \{ (a_1, \cdots, a_g) \in Y_g : F(a_1, \cdots, a_g) \text{ corresponds to a simple variety} \right \}.
$$
Then clearly 
\begin{equation} \label{eq4x}
\left | Y_g \right |=p^{\frac{1}{4}g^2(1+o(1))}
\end{equation}
and by the proof of \cite[Theorem 2.3]{Lee},
\begin{equation} \label{eq4a}
\lim_{g \rightarrow \infty}\frac{\left | Y_g^{\text{sim}} \right |}{\left | Y_g \right |} = 1. 
\end{equation}
When $(a_1, \cdots, a_g) \in Y_g^{\text{sim}}$ and $F(a_1, \cdots, a_g)$ is associated to $\pi$, then $F(a_1, \cdots, a_g)$ is the minimal polynomial of $\pi$ so 
\begin{equation} \label{eq4b}
\begin{split}
\left | \disc (R) \right | & = [R : \Z [\pi]]^{-2} \cdot \left | \disc (\Z [\pi]) \right | \\
& = [R : \Z [\pi]]^{-2} \cdot \left | \disc (F(a_1, \cdots, a_g)) \right |.
\end{split}
\end{equation}


\begin{lemma} \label{lem41}
$$
[R : \Z [\pi]] \leq p^{\frac{g(g-1)}{2}}.
$$
\end{lemma}

\begin{proof}
$R$ has a $\Z$-basis $1, \pi, \cdots, \pi^{g-1}, \overline{\pi}, \cdots, \overline{\pi}^g$. So
$$
1, \pi, \cdots, \pi^{g-1}, \overline{\pi}, \overline{\pi}^2+a_1 \overline{\pi}, \cdots, \overline{\pi}^g + a_1 \overline{\pi}^{g-1} + \cdots + a_{g-1} \overline{\pi}
$$
is also a $\Z$-basis of $R$. Now the following relations finish the proof. 
\begin{equation*}
\begin{split}
p^{g-1}\overline{\pi} & = \pi^{g-1}\overline{\pi}^g = \pi^{g-1}(-\pi^g -a_1(\pi^{g-1}+\overline{\pi}^{g-1})- \cdots -a_g) \in \Z[\pi] \\
p^{g-2}(\overline{\pi}^2 +a_1 \overline{\pi}) & = \pi^{g-2}(-\pi^g -a_1 \pi^{g-1} -a_2 (\pi^{g-2}+\overline{\pi}^{g-2})- \cdots -a_g) \in \Z[\pi] \\
& \:\: \vdots \\
\overline{\pi}^g + a_1 \overline{\pi}^{g-1} + \cdots + a_{g-1}  \overline{\pi} &= -\pi^g - a_1 \pi^{g-1} - \cdots - a_g \in \Z[\pi].
\end{split}
\end{equation*}
\end{proof}
Now we consider the lower bound of the absolute value of the discriminant of $F(a_1, \cdots, a_g)$. For a fixed $(a_1, \cdots, a_{g-1}) \in Y_g^1$, the discriminant of $F(a_1, \cdots, a_g)$ can be understood as a polynomial in $a_g$.


\begin{lemma} \label{lem42}
For a polynomial $f(X)= a_0 X^{2g} + a_1 X^{2g-1} + \cdots + a_{2g}$, the discriminant of $f$ is given by
\begin{equation} 
\disc (f) = (g^{2g}a_0^{g-1}a_{2g}^{g-1}) a_g^{2g} + (\text{lower-order terms in $a_g$})
\end{equation}
as a polynomial in $a_g$. 
\end{lemma}

\begin{proof}
$$
\disc (f) = (-1)^{\frac{2g(2g-1)}{2}} a_0^{-1} R(f,f') = (-1)^g a_0^{-1} \det A
$$
where $A$ is a $(4g-1) \times (4g-1)$ matrix given by
$$
A = 
\begin{bmatrix}
a_0 & a_1 & a_2 & \cdots & 0 & 0 & 0\\ 
0 & a_0 & a_1 & \cdots & 0 & 0 & 0\\ 
\vdots  & \vdots  & \vdots  &  & \vdots  & \vdots  & \vdots \\ 
0 & 0 & 0 & \cdots  & a_{2g-1} & a_{2g} & 0\\ 
0 & 0 & 0 & \cdots  & a_{2g-2} & a_{2g-1} & a_{2g}\\ 
2ga_0 &(2g-1)a_1 & (2g-2)a_2 & \cdots & 0 & 0 & 0\\ 
0 & 2ga_0 & (2g-1)a_1 & \cdots & 0 & 0 & 0\\ 
\vdots  & \vdots  & \vdots  &  & \vdots  & \vdots  & \vdots \\ 
0 & 0 & 0 & \cdots  & 2a_{2g-2} & a_{2g-1} & 0\\ 
0 & 0 & 0 & \cdots  & 3a_{2g-3} & 2a_{2g-2} & a_{2g-1}
\end{bmatrix}
$$
There are exactly $2g$ columns of $A$ having $a_g$ or $ga_g$. Thus $\disc (f)$, as a polynomial in $a_g$, has degree at most $2g$. Now we have to compute the coefficient of $a_g^{2g}$. Suppose that we choose $4g-1$ entries of $A$ from different rows and columns so that $2g$ of them are $a_g$ or $ga_g$. 
\begin{itemize}

\item $A_{4g-1, 3g}=ga_g$ should be chosen, because it is the only entry of the $3g$-th column of $A$ which is $a_g$ or $ga_g$. 

\item $A_{4g-1, 4g-1}$ cannot be chosen, so $A_{2g-1, 4g-1}=a_{2g}$ should be chosen. 

\item $A_{2g-1, 3g-1}=a_g$ cannot be chosen, so $A_{4g-2, 3g-1}=ga_g$ should be chosen. 

\item $A_{2g-1, 4g-2}$, $A_{4g-2, 4g-2}$ and $A_{4g-1, 4g-2}$ cannot be chosen, so $A_{2g-2, 4g-2}=a_{2g}$ should be chosen. 

\end{itemize}

Iterating this process, we choose
$$
A_{4g-i, 3g+1-i}=ga_g \,\, (1 \leq i \leq g)
$$
and 
$$
A_{2g-i, 4g-i}=a_{2g} \,\, (1 \leq i \leq g-1).
$$
\begin{itemize}

\item One of $A_{1, 1}=a_0$ and $A_{2g, 1}=2g a_0$ should be chosen. Also one of $A_{1, g+1}=a_g$ and $A_{2g, g+1}=ga_g$ should be chosen. 

\item Now $A_{1, 2}$ and $A_{2g, 2}$ cannot be chosen, so one of $A_{2,2}=a_0$ and $A_{2g+1, 2}=2ga_0$ should be chosen. Also one of $A_{2, g+2}=a_g$ and $A_{2g+1, g+2}=ga_g$ should be chosen.

\end{itemize}

Iterating this process, we choose exactly one of
$$
\left \{ A_{i,i} = a_0, \, A_{2g-1+i,i} = 2g a_0 \right \}
$$
and 
$$
\left \{ A_{2g-1+i, g+i}=ga_g, \, A_{i, g+i}=a_g \right \}
$$
for each $1 \leq i \leq g$. Now it is easy to show that
\begin{equation*}
\begin{split}
\det A & = \prod_{i=1}^{g} \begin{vmatrix}
A_{i,i} & A_{i, g+i}\\ 
A_{2g-1+i, i} & A_{2g-1+i, g+i}
\end{vmatrix} \cdot \prod_{i=1}^{g} A_{4g-i, 3g+1-i} \cdot \prod_{i=1}^{g} A_{2g-i, 4g-i} + (\text{lower-order terms in $a_g$}) \\
& = (-ga_0a_g)^g(ga_g)^ga_{2g}^{g-1} + (\text{lower-order terms in $a_g$}) \\
& = ((-1)^g g^{2g} a_0^g a_{2g}^{g-1}) a_g^{2g} + (\text{lower-order terms in $a_g$}).
\end{split}
\end{equation*}
This finishes the proof.
\end{proof}


\begin{corollary} \label{cor43}
\begin{equation} \label{eq4d}
\disc (F(a_1, \cdots, a_g)) = g^{2g} p^{g(g-1)} a_g^{2g} + (\text{lower-order terms in $a_g$})
\end{equation}
as a polynomial in $a_g$. 
\end{corollary}

Suppose that $g$ is sufficiently large. 
By the equation (\ref{eq4a}), $\left | \disc (F(a_1, \cdots, a_g)) \right |$ is large enough for sufficiently many $(a_1, \cdots, a_g) \in Y_g^{\text{sim}}$ if it is large enough for sufficiently many $(a_1, \cdots, a_g) \in Y_g$. This should be true if for any $(a_1, \cdots, a_{g-1}) \in Y_g^1$, $\left | \disc (F(a_1, \cdots, a_g)) \right |$ is large enough for sufficiently many $a_g \in Y_g^2$. Since $\disc (F(a_1, \cdots, a_g))$ can be written as equation (\ref{eq4d}), the following proposition is natural. 


\begin{proposition} \label{prop44}
For any $\varepsilon >0$, there exists $g_0=g_0(\varepsilon)>0$ such that for every $g \geq g_0$ and a monic polynomial $f$ of degree $2g$, 
$$
\mu \left ( \left \{ x \in \left [ -\frac{2p^{\frac{g}{2}}}{g}, \frac{2p^{\frac{g}{2}}}{g} \right ] : \left | f(x) \right | \leq p^{g^2(1-\varepsilon)} \right \} \right ) \leq \frac{p^{\frac{g}{2}}}{g}
$$
where $\mu (E)$ denotes the Lebesgue measure of $E$. 
\end{proposition}

\begin{proof}
Denote
$$
I_1 := \left \{ x \in \left [ -\frac{2p^{\frac{g}{2}}}{g}, \frac{2p^{\frac{g}{2}}}{g} \right ] : \left | f(x) \right | \leq p^{g^2(1-\varepsilon)} \right \}
$$
and assume that $\mu (I_1) >0$. Since the degree of $f$ is $2g$, there are
$$
-\frac{2p^{\frac{g}{2}}}{g}=x_0 < x_1 < \cdots < x_t = \frac{2p^{\frac{g}{2}}}{g} \:\: (1 \leq t \leq 2g)
$$
such that $f$ is strictly increasing or strictly decreasing on each interval $(x_i, x_{i+1})$ ($0 \leq i \leq t-1$). Since each $I_1 \cap [x_i, x_{i+1}]$ is a (possibly empty) interval, we have
$$
\int_{I_1}\left | f'(x) \right |dx = \sum_{i=0}^{t-1} \int_{I_1 \cap [x_i, x_{i+1}]}\left | f'(x) \right |dx \leq 2g \cdot 2p^{g^2(1-\varepsilon)}.
$$
Choose any $c_1 \in (0,1)$ and let
$$
I_2 := \left \{ x \in I_1 : \left | f'(x) \right | \leq \frac{2g \cdot 2p^{g^2(1-\varepsilon)}}{(1-c_1)\mu(I_1)} \right \}.
$$
Then $I_2$ is measurable and 
$$
\mu (I_2) \geq c_1 \mu (I_1).
$$
Repeating the same procedure for $f'$ and $I_2$, we have
$$
\int_{I_2}\left | f^{(2)}(x) \right |dx 
\leq \frac{2g(2g-1) \cdot 2^2p^{g^2(1-\varepsilon)}}{(1-c_1)\mu (I_1)}
$$
and for any $c_2 \in (0,1)$, there exists $I_3 \subset I_2$ such that
$$
\mu (I_3) \geq c_2 \mu (I_2) \geq c_1c_2 \mu (I_1)
$$
and
$$
\left | f^{(2)}(x) \right | \leq \frac{2g (2g-1) \cdot 2^2p^{g^2(1-\varepsilon)}}{(1-c_1)\mu(I_1) \cdot (1-c_2)\mu(I_2)}
$$
for any $x \in I_3$. Iterating this process, we get measurable subsets $I_{2g} \subset \cdots \subset I_2 \subset I_1$ such that
\begin{equation} \label{eqnew1}
\mu (I_{2g}) \geq c_{2g-1} \mu (I_{2g-1}) \geq \cdots \geq c_1c_2 \cdots c_{2g-1} \mu (I_1)
\end{equation}
and
\begin{equation} \label{eqnew2}
\begin{split}
(2g)! \mu (I_{2g}) & = \int _{I_{2g}} \left | f^{(2g)}(x) \right | dx   \\
& \leq 2^{2g}(2g)! \frac{p^{g^2(1-\varepsilon)}}{\prod_{i=1}^{2g-1}(1-c_i)\mu(I_i)} \\
& \leq 2^{2g}(2g)!p^{g^2(1-\varepsilon)} \cdot \frac{1}{\prod_{i=1}^{2g-1}(1-c_i)c_i^{2g-1-i} \mu(I_1)^{2g-1}}.
\end{split}
\end{equation}
for any $c_1, \cdots, c_{2g-1} \in (0,1)$ by the equation (\ref{eqnew1}). This gives
$$
\left ( \prod_{i=1}^{2g-1}(1-c_i)c_i^{2g-i} \right )\mu (I_1)^{2g} \leq 2^{2g} p^{g^2(1-\varepsilon)}. 
$$
If we choose $c_1, \cdots, c_{2g-1}$ by 
$$
c_1 = \cdots = c_{2g-1} = 1-\frac{1}{g},
$$
then
\begin{equation} \label{eqnew3}
\prod_{i=1}^{2g-1}(1-c_i)c_i^{2g-i} = \frac{1}{g^{2g-1}}(1-\frac{1}{g})^{g(2g-1)} >\frac{1}{(3g)^{2g-1}}
\end{equation}
for sufficiently large $g$. Now there exists $g_0>0$ (depends only on $\varepsilon$) such that
$$
\mu (I_1) \leq 6g p^{\frac{g}{2} (1-\varepsilon)} \leq \frac{p^{\frac{g}{2}}}{g}
$$
for any $g \geq g_0$. 
\end{proof}


For each $b = (b_1, \cdots, b_{g-1}) \in Y_g^1$, define
$$
P_{b}(X) := \disc (F(b_1, \cdots, b_{g-1}, X)).
$$
Then by Corollary \ref{cor43}, 
$$
P_{b}(X) = g^{2g}p^{g(g-1)}Q_{b}(X)
$$
for some monic polynomial $Q_{b}$ of degree $2g$. Applying Proposition \ref{prop44} to the polynomial $Q_{b}$, we obtain that for any $\varepsilon >0$, the set
$$
A_{b, \varepsilon} = \left \{ x \in  \left [ -\frac{2p^{\frac{g}{2}}}{g}, \frac{2p^{\frac{g}{2}}}{g} \right ] : \left | P_b(x) \right | \geq g^{2g} p^{2g^2-g-g^2 \varepsilon} \right \}
$$
satisfies
$$
\mu (A_{b, \varepsilon}) \geq \frac{3 p^{\frac{g}{2}}}{g}
$$
for sufficiently large $g$. For any $\varepsilon >0$, denote
\begin{equation*}
\begin{split}
S_{g, \varepsilon} & := \left \{ \mathbf{a} \in Y_g : \left | \disc (F(\mathbf{a})) \right | \geq g^{2g}p^{2g^2-g-g^2\varepsilon} \right \} \\
S_{g, \varepsilon}^{\text{sim}} & := S_{g, \varepsilon} \cap Y_g^{\text{sim}} = \left \{ \mathbf{a} \in Y_g^{\text{sim}} : \left | \disc (F(\mathbf{a})) \right | \geq g^{2g}p^{2g^2-g-g^2\varepsilon} \right \}.
\end{split}
\end{equation*}


\begin{theorem} \label{thm45}
For any $\varepsilon >0$, 
\begin{equation} \label{eq4f}
\left | S_{g, \varepsilon}^{\text{sim}} \right | = p^{\frac{1}{4}g^2(1+o(1))}.
\end{equation}
\end{theorem}

\begin{proof}
As in the proof of Proposition \ref{prop44}, there are
$$
-\frac{2p^{\frac{g}{2}}}{g}=x_0 < x_1 < \cdots < x_t = \frac{2p^{\frac{g}{2}}}{g} \:\: (1 \leq t \leq 2g)
$$
such that $P_{b}$ is strictly increasing or strictly decreasing on each $(x_i, x_{i+1})$. Thus $A_{b, \varepsilon}$ is a disjoint union of closed intervals
$$
I_1, \cdots, I_s \:\: (s \leq t+1 \leq 2g+1)
$$
so
\begin{equation*} 
\begin{split}
\left | A_{b, \varepsilon} \cap (\Z -p \Z) \right | & = \left | A_{b, \varepsilon} \cap \Z \right | - \left | A_{b, \varepsilon} \cap p \Z \right |   \\
& \geq \sum_{i=1}^{s} \left | I_i \cap \Z \right | - \left | \left [ -\frac{2p^{\frac{g}{2}}}{g}, \frac{2p^{\frac{g}{2}}}{g} \right ] \cap p \Z \right |  \\
& \geq \sum_{i=1}^{s} (\mu (I_i)-1) - \left ( \frac{2p^{\frac{g}{2}}}{g} + 1 \right ) \\
& \geq \mu(A_{b, \varepsilon}) - (2g+1) - \left ( \frac{2p^{\frac{g}{2}}}{g} + 1 \right ) \\
& \geq \frac{p^{\frac{g}{2}}}{g} - (2g+2).
\end{split}
\end{equation*}
for sufficiently large $g$. Now
\begin{equation*} 
\begin{split}
\frac{\left | S_{g, \varepsilon} \right |}{\left | Y_g \right |} & \geq \min_{b \in Y_g^1} \frac{\left | \left \{ x \in Y_g^2 : (b,x) \in S_{g, \varepsilon} \right \} \right |}{\left | Y_g^2 \right |}  \\
& = \min_{b \in Y_g^1} \frac{\left | A_{b, \varepsilon} \cap (\Z - p \Z) \right |}{\left | Y_g^2 \right |}  \\
& \geq \frac{\frac{p^{\frac{g}{2}}}{g} - (2g+2)}{\frac{4p^{\frac{g}{2}}}{g}+1} \\
& > \frac{1}{5}
\end{split}
\end{equation*}
for sufficiently large $g$. Combining this with the equation (\ref{eq4a}), we have
$$
\frac{\left | S_{g, \varepsilon}^{\text{sim}} \right |}{\left | Y_g \right |} \geq \frac{\left | S_{g, \varepsilon} \right |}{\left | Y_g \right |} + \frac{\left | Y_g^{\text{sim}} \right |}{\left | Y_g \right |}-1 > \frac{1}{10}
$$
for sufficiently large $g$. The equation (\ref{eq4x}) finishes the proof. 
\end{proof}

For $\mathbf{a} \in Y_g^{\text{sim}}$, denote a Weil $p$-number which corresponds to the Weil $p$-polynomial $F(\mathbf{a})$ by $\pi_{\mathbf{a}}$ and let $R_{\mathbf{a}} := \Z [\pi_{\mathbf{a}}, \overline{\pi_{\mathbf{a}}} ]$. Define 
$$
T_{g, \varepsilon}^{\text{sim}} := \left \{ \mathbf{a} \in Y_g^{\text{sim}} : \left | \disc (R_{\mathbf{a}}) \right | \geq g^{2g}p^{g^2 (1 - \varepsilon)} \right \}.
$$
The following corollary comes from the equation (\ref{eq4b}), Lemma \ref{lem41} and Theorem \ref{thm45}. 


\begin{corollary} \label{cor46}
For any $\varepsilon >0$, 
\begin{equation} \label{eq4g}
\left | T_{g, \varepsilon}^{\text{sim}} \right | = p^{\frac{1}{4}g^2(1+o(1))}.
\end{equation}
\end{corollary}

\subsection{Class number of $R$} \label{Sub42}

In this section we consider the lower bound of the class number $h(R)$ of $R$. To give a lower bound of $h(R)$ from a lower bound of $\disc (R)$, we need a relation between them. Brauer-Siegel theorem provides a relation between $h (\mathcal{O}_E)$ and $\sqrt{\left | \disc (\mathcal{O}_E) \right |}$, but it is ineffective. For CM-fields, Stark \cite{Sta74} provided an effective result which is a weaker version of the Brauer-Siegel theorem. 

\begin{theorem} \label{thm47}
(\cite{Sta74}, Theorem 2) Let $K$ be a CM-field of degree $2n$ containing a totally real subfield $k$ of degree $n$ and $f$ be a positive integer given by $\left | \disc (\mathcal{O}_K) \right | = \disc (\mathcal{O}_k)^2 f$. For $\varepsilon$ in the range $0< \varepsilon \leq \frac{1}{2}$, there is an effectively computable constant $c(\varepsilon) >0$ such that
\begin{equation} \label{eq42a}
\begin{split}
h(K) & > h(k) \frac{1}{n \cdot n!} c(\varepsilon)^n \left | \disc (\mathcal{O}_k) \right |^{\frac{1}{2}-\frac{1}{n}-\varepsilon} f^{\frac{1}{2}-\frac{1}{2n}} \\
& \geq h(k) \frac{1}{n \cdot n!} c(\varepsilon)^n \left | \frac{\disc (\mathcal{O}_K)}{\disc (\mathcal{O}_k)} \right |^{\frac{1}{2}-\frac{1}{n}-\varepsilon}. 
\end{split}
\end{equation}
\end{theorem}

\begin{remark} \label{rmk48}
In the original statement of Theorem 2 in \cite{Sta74}, $h(k)$ does not appear in the inequality. One can easily check that the term $h(k)$ can be added in the right-hand side of the inequality (\ref{eq42a}) by following the proof of \cite[Theorem 2]{Sta74} which is a direct consequence of \cite[Theorem 2$'$]{Sta74}. 
\end{remark}


Motivated by the result of Stark, we consider a lower bound of $\frac{h(R)}{h(R^{+})}$ rather than a lower bound of $h(R)$. It turns out that this is essential for our proof. First we review some formulas relating the class numbers and the discriminants of $R$, $R^+$, $\mathcal{O}_E$ and $\mathcal{O}_{E^+}$. 

For an order $\mathcal{O}$ of a number field $F$, denote
$$
\widehat{\mathcal{O}} := \bigoplus_{\mathfrak{p}} \mathcal{O}_{\mathfrak{p}}
$$
where $\mathfrak{p}$ runs through the nonzero prime ideals of $\mathcal{O}$. By \cite[Proposition 1.12.9]{Neu99}, 
$$
h(R) = h(\mathcal{O}_E) \cdot \frac{[\widehat{\mathcal{O}_E}^{\times} : \widehat{R}^{\times}]}{[{\mathcal{O}^{\times}_E} : R^{\times}]} \; \text{  and  } \;
h(R^+) = h(\mathcal{O}_{E^+}) \cdot \frac{[\widehat{\mathcal{O}_{E^+}}^{\times} : \widehat{R^+}^{\times}]}{[{\mathcal{O}^{\times}_{E^+}} : {(R^+)}^{\times}]}.
$$
We also have 
$$
\left | \disc (R) \right | 
= [\mathcal{O}_E : R]^2 \cdot \left | \disc (\mathcal{O}_E) \right |
$$
and
$$
\left | \disc (R^+) \right | 
= [\mathcal{O}_{E^+} : R^+]^2 \cdot \left | \disc (\mathcal{O}_{E^+}) \right |.
$$


\begin{lemma} \label{lem49}
\begin{equation} \label{eq42b}
\left | \disc (R^+) \right | \leq p^{\frac{g^2}{2}(1+o(1))}.
\end{equation}
\end{lemma}

Before giving a proof, we introduce the notion of transfinite diameter. The \textit{transfinite diameter} of a compact set $A \subset \mathbb{C}$ is defined by
\begin{equation} \label{eqnew4}
\gamma (A) := \lim_{n \rightarrow \infty} \max_{z_1, \cdots, z_n \in A} \left ( \prod_{1 \leq i <j \leq n} \left | z_j - z_i \right | \right )^{\frac{2}{n(n-1)}}.
\end{equation}
It is proved by Fekete \cite[Section 6]{Fek23} that 
\begin{equation} \label{eqnew6}
\gamma([0,1]) = \frac{1}{4}.
\end{equation}

\begin{proof}
For a polynomial $f(x) = F(a_1, \cdots, a_g)$, 
$$
\frac{f(x)}{x^g} = \left (x^g + \left (\frac{p}{x}  \right )^g  \right ) +a_1 \left ( x^{g-1}+\left (\frac{p}{x}  \right )^{g-1} \right ) + \cdots + a_g = h \left ( x+\frac{p}{x} \right )
$$
for some monic polynomial $h$ of degree $g$. Since $f$ is the minimal polynomial of $\pi$, $h$ is the minimal polynomial of $\pi + \overline{\pi}$ so
$$
\disc (R^+) = \disc (h).
$$
Any root $\alpha$ of $h$ is of the form $\pi_1 + p \pi_1^{-1}$ for some root $\pi_1$ of $f$. Since the absolute value of $\pi_1$ is $\sqrt{p}$, $\alpha = \pi_1 + \overline{\pi_1} \in \R$ and $\alpha \in [-2 \sqrt{p}, 2 \sqrt{p}]$. Now $\disc (h)$ is given by
$$
\prod_{1 \leq i < j  \leq g} (x_j-x_i)^2
$$
for some real numbers
$$
-2 \sqrt {p} \leq x_1 \leq \cdots \leq x_g \leq 2 \sqrt{p}.
$$
Denote
$$
y_i = \frac{x_i}{4 \sqrt{p}} + \frac{1}{2} \:\: (1 \leq i \leq g).
$$
Then $y_i \in [0,1]$ for each $1 \leq i \leq g$ and
\begin{equation} \label{eqnew5}
\begin{split}
\prod_{1 \leq i < j  \leq g} (x_j-x_i)^2 & = (4 \sqrt{p})^{2 \binom{g}{2}} \prod_{1 \leq i < j  \leq g} (y_j-y_i)^2 \\
& = p^{\frac{1}{2}g^2(1+o(1))}2^{2g^2} \prod_{1 \leq i < j  \leq g} (y_j-y_i)^2.
\end{split}
\end{equation}
Now it is enough to show that
$$
\prod_{1 \leq i < j  \leq g} (y_j-y_i)^2 \leq \frac{1}{2^{2g^2(1+o(1))}}, 
$$
which is a direct consequence of the equation (\ref{eqnew6}).
\end{proof}


\begin{proposition} \label{prop410}
For any $\varepsilon \in (0, \frac{1}{2} )$, 
\begin{equation} \label{eq42c}
\frac{h(R)}{h(R^+)} \geq \frac{\left | \disc (R) \right | ^{\frac{1}{2}-\frac{1}{g}-\varepsilon}}{2^{g^2 \varepsilon }p^{\frac{1}{4}g^2(1+2 \varepsilon + o(1))}}.
\end{equation}
\end{proposition}

\begin{proof}
Let $c(\varepsilon)>0$ be a constant as in Theorem \ref{thm47}. By Theorem \ref{thm47} and the formulas above, 
\begin{equation*}
\begin{split}
\frac{h(R)}{h(R^+)} & = \frac{h(\mathcal{O}_E) \cdot \frac{[\widehat{\mathcal{O}_E}^{\times} : \widehat{R}^{\times}]}{[{\mathcal{O}^{\times}_E} : R^{\times}]}}{h(\mathcal{O}_{E^+}) \cdot \frac{[\widehat{\mathcal{O}_{E^+}}^{\times} : \widehat{R^+}^{\times}]}{[{\mathcal{O}^{\times}_{E^+}} : {(R^+)}^{\times}]}} \\
& > \frac{1}{g \cdot g!} c(\varepsilon)^g 
\cdot \left | \frac{\disc (\mathcal{O}_E)}{\disc (\mathcal{O}_{E^+})} \right |^{\frac{1}{2}-\frac{1}{g}-\varepsilon} 
\cdot \frac{[\widehat{\mathcal{O}_E}^{\times} : \widehat{R}^{\times}]}{[{\mathcal{O}^{\times}_E} : R^{\times}]} 
\cdot \frac{[{\mathcal{O}^{\times}_{E^+}} : {(R^+)}^{\times}]}{[\widehat{\mathcal{O}_{E^+}}^{\times} : \widehat{R^+}^{\times}]} \\
& = \frac{1}{g \cdot g!} c(\varepsilon)^g 
\cdot \left | \frac{\disc (R)}{\disc (R^+)} \right |^{\frac{1}{2}-\frac{1}{g}-\varepsilon} 
\left ( \frac{[ \mathcal{O}_{E^+} : R^+ ]}{[ \mathcal{O}_E : R ]} \right )^{1-\frac{2}{g}-2 \varepsilon}
\cdot \frac{[{\mathcal{O}^{\times}_{E^+}} : {(R^+)}^{\times}]}{[{\mathcal{O}^{\times}_E} : R^{\times}]} 
\cdot \frac{[\widehat{\mathcal{O}_E}^{\times} : \widehat{R}^{\times}]}{[\widehat{\mathcal{O}_{E^+}}^{\times} : \widehat{R^+}^{\times}]}.
\end{split}
\end{equation*}
Since
$$
 \frac{1}{g \cdot g!} c(\varepsilon)^g 
\cdot \left | \frac{\disc (R)}{\disc (R^+)} \right |^{\frac{1}{2}-\frac{1}{g}-\varepsilon}  \geq \frac{ \left | \disc (R) \right | ^{\frac{1}{2}-\frac{1}{g}-\varepsilon}}{p^{\frac{1}{4}g^2(1+o(1))}}
$$
by Lemma \ref{lem49}, it is enough to show that
\begin{equation} \label{eq42d}
\left ( \frac{[ \mathcal{O}_{E^+} : R^+ ]}{[ \mathcal{O}_E : R ]} \right )^{1-\frac{2}{g}-2 \varepsilon}
\cdot \frac{[{\mathcal{O}^{\times}_{E^+}} : {(R^+)}^{\times}]}{[{\mathcal{O}^{\times}_E} : R^{\times}]} 
\cdot \frac{[\widehat{\mathcal{O}_E}^{\times} : \widehat{R}^{\times}]}{[\widehat{\mathcal{O}_{E^+}}^{\times} : \widehat{R^+}^{\times}]} 
\geq \frac{1}{2^{g^2 \varepsilon} p^{\frac{1}{2}g^2 (\varepsilon +o(1))}}.
\end{equation}
We divide it to three parts and prove that each part is not too small. 

\begin{enumerate}


\item By \cite[Theorem 4.12]{Was97}, 
$$
\frac{[\mathcal{O}^{\times}_{E^+} : {(R^+)}^{\times}]}{[\mathcal{O}^{\times}_E : R^{\times}]} \geq \frac{1}{[\mathcal{O}^{\times}_E : \mathcal{O}^{\times}_{E^+}]} \geq \frac{1}{2 \left |\mu_E  \right |}
$$
where $\mu_E$ is the group of roots of unity in $E$. If $\left |\mu_E  \right | = r$, then 
$$
2g = [E: \Q] \geq [\Q (\zeta_r) : \Q] = \phi(r) \geq \frac{\sqrt{r}}{2}
$$
so $2 \left |\mu_E  \right | =2r \leq 32g^2$. 


\item By Lemma \ref{lem49}, 
\begin{equation*}
\begin{split}
\left ( \frac{[ \mathcal{O}_{E^+} : R^+ ]}{[ \mathcal{O}_E : R ]} \right )^{-\frac{2}{g}-2 \varepsilon} &  \geq \frac{1}{[\mathcal{O}_{E^+} : R^+]^{\frac{2}{g} + 2 \varepsilon}} \\
& \geq \frac{1}{\left |  \disc (R^+) \right | ^{\frac{1}{g}+ \varepsilon}} \\
& \geq \frac{1}{(2^{g^2}p^{\frac{g^2}{2}(1+o(1))})^{\frac{1}{g} + \varepsilon}} \\
& = \frac{1}{2^{g^2 \varepsilon} p^{\frac{1}{2}g^2 (\varepsilon +o(1))}}.
\end{split}
\end{equation*}


\item Since $[ \mathcal{O}_E : R ] = [ \widehat{\mathcal{O}_E} : \widehat{R} ]$ and $[ \mathcal{O}_{E^+} : R^+ ] = [ \widehat{\mathcal{O}_{E^+}} : \widehat{R^+} ]$, we have
$$
\frac{[ \mathcal{O}_{E^+} : R^+ ]}{[ \mathcal{O}_E : R ]}  \cdot \frac{[\widehat{\mathcal{O}_E}^{\times} : \widehat{R}^{\times}]}{[\widehat{\mathcal{O}_{E^+}}^{\times} : \widehat{R^+}^{\times}]} 
 = \frac{[\widehat{\mathcal{O}_E}^{\times} : \widehat{R}^{\times}]}{[\widehat{\mathcal{O}_E} : \widehat{R}]} 
 \cdot \frac{[\widehat{\mathcal{O}_{E^+}} : \widehat{R^+}]}{[\widehat{\mathcal{O}_{E^+}}^{\times} : \widehat{R^+}^{\times}]}.
$$
Let $\mu_{\mathfrak{p}}$ be a usual Haar measure on a local field $E_{\mathfrak{p}}$. Then
\begin{equation*}
\begin{split}
\frac{[\widehat{\mathcal{O}_E}^{\times} : \widehat{R}^{\times}]}{[\widehat{\mathcal{O}_E} : \widehat{R}]}
& = \prod_{\mathfrak{p} \mid \disc (R)} \frac{[\mathcal{O}^{\times}_{E, \mathfrak{p}} : R^{\times}_{\mathfrak{p}}]}{[\mathcal{O}_{E, \mathfrak{p}} : R_{\mathfrak{p}}]} \\
& = \prod_{\mathfrak{p} \mid \disc (R)} \frac{\mu_{\mathfrak{p}}(\mathcal{O}^{\times}_{E, \mathfrak{p}})}{\mu_{\mathfrak{p}}(R^{\times}_{\mathfrak{p}})} \cdot \frac{\mu_{\mathfrak{p}}(R_{\mathfrak{p}})}{\mu_{\mathfrak{p}}(\mathcal{O}_{E, \mathfrak{p}})} \\
& \geq \prod_{\mathfrak{p} \mid \disc (\Z [\pi])} \frac{\mu_{\mathfrak{p}}(\mathcal{O}^{\times}_{E, \mathfrak{p}})}{\mu_{\mathfrak{p}}(\mathcal{O}_{E, \mathfrak{p}})} \\
& = \prod_{\mathfrak{p} \mid \disc (\Z [\pi])} \left ( 1-\frac{1}{\kappa (\mathfrak{p})} \right )
\end{split}
\end{equation*}
for $\kappa (\mathfrak{p}) := \left | \mathcal{O}_E / \mathfrak{p} \right |$. (Note that if $\mathfrak{p} \nmid \disc (R)$, then $\mathfrak{p} \nmid [\mathcal{O}_E : R]$ so $\mathcal{O}_{E, \mathfrak{p}} = R_{\mathfrak{p}}$.) Denote the $j$-th prime by $p_j$. By Lemma \ref{lem3a}, 
$$
\left | \disc (\Z [\pi]) \right | = \left | \disc (f) \right | \leq (2g)^{2g} \sqrt{p}^{2 \binom{2g}{2}} < (2gp)^{2g^2}
$$
so the number of primes in $\Z$ which divide $\disc (\Z [\pi])$ is less than $2gp+2g^2$ because
$$
\prod_{j=2gp+1}^{2gp+2g^2}p_j > \prod_{j=2gp+1}^{2gp+2g^2} 2gp = (2gp)^{2g^2}.
$$
For each prime in $\Z$, there are at most $2g$ primes in $\mathcal{O}_E$ lying above it. Thus
\begin{equation*}
\begin{split}
\prod_{\mathfrak{p} \mid \disc (\Z [\pi])} \left ( 1-\frac{1}{\kappa (\mathfrak{p})} \right )
& > \prod_{j=1}^{2gp+2g^2} \left ( 1-\frac{1}{p_j} \right )^{2g} \\
& \geq \prod_{j=1}^{2gp+2g^2} \left ( 1-\frac{1}{j+1} \right )^{2g} \\ 
& = \frac{1}{(2gp+2g^2+1)^{2g}} \\
& = \frac{1}{p^{o(g^2)}}
\end{split}
\end{equation*}
By the same reason, 
$$
\frac{[\widehat{\mathcal{O}_{E^+}} : \widehat{R^+}]}{[\widehat{\mathcal{O}_{E^+}}^{\times} : \widehat{R^+}^{\times}]}
\geq \prod_{\mathfrak{p} \mid \disc ( R^+)} \frac{\mu_{\mathfrak{p}}((R^+_{\mathfrak{p}})^{\times})}{\mu_{\mathfrak{p}}(R^{+}_{\mathfrak{p}})} \geq \frac{1}{p^{o(g^2)}}.
$$
\end{enumerate}

It is clear that (1), (2) and (3) imply the inequality (\ref{eq42d}).
\end{proof}

Now we are ready to prove the main theorem of the paper. 

\begin{theorem} \label{thm411}
\begin{equation} \label{eq42e}
B(p,g) \geq B^{\text{sim}}(p,g) \geq p^{\frac{1}{2}g^2(1+o(1))}. 
\end{equation}
\end{theorem}

\begin{proof}
For $\mathbf{a} \in Y_g^{\text{sim}}$, denote a Weil $p$-number which corresponds to the Weil $p$-polynomial $F(\mathbf{a})$ by $\pi_{\mathbf{a}}$ and let $R_{\mathbf{a}} := \Z [\pi_{\mathbf{a}}, \overline{\pi_{\mathbf{a}}} ]$. Then for any $\varepsilon_1, \varepsilon_2 \in (0, \frac{1}{2})$, 
\begin{equation} \label{eq42f}
\begin{split}
B^{\text{sim}}(p,g) & \geq \sum_{\mathbf{a} \in Y_g^{\text{sim}}} \left | \Isog (\pi_{\mathbf{a}}) \right | \\
& \geq \sum_{\mathbf{a} \in Y_g^{\text{sim}}} h(R_{\mathbf{a}}) \:\: (\text{Corollary } \ref{cor22}) \\
& \geq \sum_{\mathbf{a} \in T_{g, \varepsilon_1}^{\text{sim}}} h(R_{\mathbf{a}}) \\
& \geq \frac{1}{2^{g^2 \varepsilon_2}p^{\frac{1}{4}g^2(1+2 \varepsilon_2 + o(1))}} 
\cdot \sum_{\mathbf{a} \in T_{g, \varepsilon_1}^{\text{sim}}} \left | \disc (R_{\mathbf{a}}) \right | ^{\frac{1}{2}-\frac{1}{g}-\varepsilon_2} \:\: (\text{Proposition } \ref{prop410}) \\
& \geq \frac{1}{2^{g^2 \varepsilon_2}p^{\frac{1}{4}g^2(1+2 \varepsilon_2 + o(1))}} 
\cdot \left | T_{g, \varepsilon_1}^{\text{sim}} \right | \cdot (g^{2g}p^{g^2 (1 - \varepsilon_1)})^{\frac{1}{2} -\frac{1}{g} - \varepsilon_2} \\
& = \frac{1}{2^{g^2 \varepsilon_2}} \cdot p^{g^2\left ( -\frac{1}{2} \varepsilon_2 + (1 - \varepsilon_1)\left ( \frac{1}{2} - \varepsilon_2 \right ) +o(1) \right )} \:\: (\text{Corollary } \ref{cor46}) .
\end{split}
\end{equation}
If we let $\varepsilon_1, \varepsilon_2 \rightarrow 0$ in (\ref{eq42f}), then the conclusion follows.
\end{proof}


\section*{Acknowledgments}

This work was partially supported by Samsung Science and Technology Foundation (SSTF-BA1802-03). 
The author is very grateful to National Center for Theoretical Sciences (NCTS) in Taiwan for support and hospitality, especially to Chia-Fu Yu for his helpful comments, suggestions and encouragement during the visit to NCTS. 
The author thank Sungmun Cho, Michael Lipnowski, Stefano Marseglia and Jacob Tsimerman for their helpful comments. The author also deeply thank the anonymous referee for their comment about arithmetic capacity which greatly improved the main result of the paper and other comments that improved the exposition of the paper. 


\vspace{3mm}

\footnotesize{
\textsc{Jungin Lee: Department of Mathematics, Pohang University of Science and Technology, Pohang, Gyeongbuk, Republic of Korea 37655.} 

\textit{E-mail address}: \changeurlcolor{black}\href{mailto:jilee.math@gmail.com}{jilee.math@gmail.com} 

\end{document}